\documentclass[11pt]{article}
\usepackage{geometry}                
\usepackage[usenames,dvipsnames,svgnames,table]{xcolor}    
\usepackage{graphicx}
\usepackage{lineno}
\usepackage{amssymb,amsfonts, amsmath, amsthm}
\usepackage{epstopdf,enumitem}
\usepackage[normalem]{ulem}
\usepackage{thmtools}
\title{A note on clustered cells}

\author{Saskia Chambille, Pablo Cubides Kovacsics and Eva Leenknegt}
\newtheorem*{theorem*}{Theorem}

\declaretheorem[style=definition,qed=$\blacksquare$,numberwithin=section]{definition}

\declaretheorem[style=definition,qed=$\blacksquare$, sibling = definition]{lemma-definition}

\declaretheorem[style=theorem, sibling = definition]{theorem}
\declaretheorem[style=theorem, sibling = definition]{lemma}

\declaretheorem[style=theorem, sibling = definition]{question}

\newcommand{\Q}{\mathbb{Q}}
\newcommand{\QQ}{\Q}

\newcommand{\Z}{\mathbb{Z}}

\newcommand{\N}{\mathbb{N}}

\newcommand{\ord}{\mathrm{ord}}

\newcommand{\ac}{\text{ac}\,}
\newcommand{\acm}{\mathrm{ac}_{m}\,}

\begin{document}


\maketitle

\begin{abstract}
This note contains additions to the paper \emph{Clustered cell decomposition in $P$-minimal structures}. We discuss a question which was raised in that paper, on the order of clustered cells. We also consider a notion of cells of minimal order, which is a slight optimalisation of the theorem from the original paper. \end{abstract}

 \noindent \textit{This note is a companion to \cite{cham-cub-leen}, and we refer the reader to that paper for definitions, notation and terminology.} 
\\\\
In the paper \emph{Clustered cell decomposition in $P$-minimal structures} \cite{cham-cub-leen}, we proved a cell decomposition theorem for general $P$-minimal structures (without the requirement of definable Skolem functions). The current note is motivated by the question whether there exists an \emph{optimal} version of this theorem. In asking this, we are fully aware that \emph{optimality} is not a very well defined notion, and hence part of this note will be devoted to proposing a possible interpretation of this concept, by introducing a notion of cells of minimal order in Section \ref{sec:min}.
\\\\
The following question was raised in \cite{cham-cub-leen}:

\begin{question}\label{q1} 
Can every regular clustered cell of finite order be partitioned into finitely many regular clustered cells of order 1?
\end{question}

Note that, if this question were to have an affirmative answer, this would imply a significant simplification of the Clustered Cell Decomposition Theorem from \cite{cham-cub-leen}. 
Moreover, this would mean that, at least in spirit, such a  generalized cell decomposition theorem stays very close to the spirit of \emph{classical} (Denef-type) cell decomposition: recall that the absence of Skolem functions has forced us to introduce a notion of cells where the centers are no longer given by definable functions.  However, for a clustered cell of order 1, the set $\Sigma$ can still be seen as the graph of a definable function $c:S \to \mathbb{B}$, where $\mathbb{B}$ denotes the set of balls in $K$. 
Hence, a restriction to clustered cells of order 1 would then indeed be optimal in the sense that we stay as close as possible to the idea of centers as definable functions.

 Unfortunately, it may not be possible to achieve this. In Section \ref{sec:questions}, we will further explore the above question, along with some reformulations where we make a connection with cells of minimal order and a weak version of Skolem functions. 
In Section \ref{sec:biggerorder}, we will explore further properties of the tree structure of the sets $\Sigma$ associated to cells of minimal order. 
\\\\
To conclude this introduction, we will restate the Cell Decomposition Theorem from 
\cite{cham-cub-leen}, and give a short recap of some of the main notions from \cite{cham-cub-leen}.

\begin{theorem}[Clustered Cell Decomposition]\label{thm:celldecomposition}
Let $X \subseteq S \times K$ be a set definable in a $P$-minimal structure. Then there exist $n,m \in \N\backslash\{0\}$ and a finite partition of $X$ into definable sets $X_i \subseteq S_i \times K$ of one of the following forms
\begin{itemize}
\item[(i)] Classical cells 
\[X_i= \{ (s,t) \in S_i \times K \mid \alpha_{i}(s) \ \square_1 \ \ord(t-c_{i}(x)) \ \square_2 \ \beta_{i}(s) \wedge t - c_{i}(s) \in \lambda_{i} Q_{n,m} \},\]
where $\alpha_{i}, \beta_{i}$ are definable functions $S_i \to \Gamma_K$, the squares $\square_1,\square_2$ may denote either $<$ or    $\emptyset$ (i.e. `no condition'), and $\lambda_{i} \in K$. The center $c_{i}: S_i \to K$ is a definable function (which may not be unique). 
\item[(ii)] Regular clustered cells $X_i=C_{i}^{\Sigma_{i}}$ of order $k_{i}$. \item[] Let $\sigma_1, \ldots, \sigma_{k_{i}}$ be (non-definable) sections of the definable multi-ball $\Sigma_{i} \subseteq S_i \times K$, such that for each $s \in S_i$, the set $\{\sigma_1(s), \ldots, \sigma_{k_{i}}(s)\}$ contains representatives of all $k_{i}$ disjoint balls covering $(\Sigma_{i})_s$. Then $X_{i}$ partitions as
\[X_{i} = C_{i}^{\sigma_1} \cup \ldots \cup C_{i}^{\sigma_{k_{i}}},\]
where each set $C_{i}^{\sigma_l}$ is of the form
\[ C_{i}^{\sigma_l} = \{(s,t) \in S_i \times K \mid \alpha_{i}(s) \ < \ \ord(t-\sigma_l(s)) \ < \ \beta_{i}(s) \wedge t-\sigma_l(s) \in \lambda_{i} Q_{n,m}\}.\]
Here $\alpha_{i}, \beta_{i}$ are definable functions $S_i \to \Gamma_K$, $\lambda_{i} \in K \backslash \{0\}$, and $\ord\, \alpha_{i}(s) \geqslant \ord \,\sigma_l(s)$ for all $s \in S_i$. Finally, we may suppose no section of $\Sigma_i$ is definable.  
\end{itemize}
\end{theorem}



Classical cells are similar to the type of cells used in earlier results by Denef and Mourgues, with a definable function as center. For regular clustered cells, the center is given by a multi-ball of finite order. Recall that a multi-ball (of order $k$) is a definable set where each fiber $(\Sigma_i)_s$ is the disjoint union of $k$ balls of the same radius, which, in the context of cell decomposition, represent equivalence classes of centers. 

For more details we refer to \cite{cham-cub-leen}. In particular, the following definitions may be of relevance to the contents discussed in this note: Definitions 1.4 (leaf), 4.1 (equivalence class), 4.3 (branching height), 4.4 (signature), 5.3 (large/small cells), 6.1 (multi-ball), 6.2 (regular clustered cell). 

\section{A decomposition into cells of minimal order}\label{sec:min}



\begin{definition} A regular clustered cell $C^{\Sigma}$ of order $k$ is of \emph{minimal order} if it cannot be partitioned as a finite  union of regular clustered cells 
$C_i^{\Sigma_i}$ of order $k_i<k$.   
\end{definition}

Some remarks are in order here. In this definititon we allow for the option that, given a regular clustered cell $C^\Sigma$ of order $k$, there may exist a cell condition $C_1$ and a multi-ball $\Sigma_1$ such that $C^{\Sigma} = C_1^{\Sigma_1}$, but the order of $C_1^{\Sigma_1}$ is strictly lower than the order of $C^{\Sigma}$. Also in more general cases their need not be a direct connection between the original $C$ and $\Sigma$ and the $C_i$ and $\Sigma_i$ occurring in the partition.
Further, let us make the following convention: 
for the remainder of the paper, the word `partition' should be read as meaning `\emph{finite} partition'.
\\\\
 Our intention in this section is to rewrite the clustered cell decomposition theorem in terms of cells of minimal order. 
We first prove the following lemma.  

%


\begin{lemma}\label{lem:decomMinimal} Every regular clustered cell of finite order can be partitioned into regular clustered cells of minimal order. 
\end{lemma} 
 
\begin{proof}  
Let $k$ be the minimal integer for which there exists a regular clustered cell $C^\Sigma$ of order $k$ that cannot be partitioned as a finite union of regular clustered cells of minimal order. In particular, $C^\Sigma$ is not of minimal order, hence it can be partitioned into finitely many regular clustered cells $C_1^{\Sigma_1},\ldots,C_n^{\Sigma_n}$ each of order $k_i<k$. But by the minimality of $k$, each $C_i^{\Sigma_i}$ can be partitioned into regular clustered cells of minimal order, which provides a decomposition of $C^\Sigma$, contradicting the assumption. 
\end{proof}

Note that there is no canonicity here: it may well be that, by making different choices in each steps of the induction, one can obtain different partitions of the same set where the number of cells in the decomposition and their specific orders $k_i$ may differ.
Using Lemma \ref{lem:decomMinimal}, we obtain the cell decomposition result stated in the theorem below. Note that, in order to obtain this, some of the statement of Theorem \ref{thm:celldecomposition} needs to be relaxed slightly. Specifically, we can no longer require that all clustered cells occurring in the decomposition, are described using the same set $Q_{n,m}$. 

\begin{theorem}\label{thm:clustered2} Let $X\subseteq S\times K$ be a definable set. Then $X$ decomposes as a finite union of classical cells and regular clustered cells of minimal order. Moreover, no regular clustered cell has a definable section. 
\end{theorem}  

\begin{proof} By Theorem \ref{thm:celldecomposition} it suffices to show the result for a regular clustered cell $C^{\Sigma}$ which has no definable section. By Lemma \ref{lem:decomMinimal}, $C^{\Sigma}$ can be decomposed into finitely many regular clustered cells of minimal order $C_1^{\Sigma_1},\ldots,C_n^{\Sigma_n}$. It remains to check whether these cells admit a definable section. Note that if $\Sigma_i$ is of order $k_i>1$, then it cannot contain a definable section. Indeed, if such a section were to exist, this would contradict the minimality of $k_i$, since it would be possible to definably split $C_i^{\Sigma_i}$ into a regular clustered cell of order $1$ and a cell of order $k_i-1$ (see also Definition 3.9 from \cite{cham-cub-leen}).  Hence, if $\Sigma_i$ has a definable section it must be of order 1. Put $I:=\{1, \ldots, n\}$ and let $I_0:=\{i\in I \mid \Sigma_i \text{ has a definable section } \sigma_i\}$. Then $C^\Sigma$ is decomposed into
\[
\bigcup_{i\in I_0} C_i^{\sigma_i} \cup \bigcup_{i\in I\setminus I_0} C_i^{\Sigma_i},
\]
which shows the result. 
\end{proof}
Note that the uniformity for $n,m$ from Theorem \ref{thm:celldecomposition} is lost here because a priori, there is no guarantee that the cell conditions $C_i$ in the above proof are defined using the same set $Q_{n,m}$. In the proof of Theorem \ref{thm:celldecomposition}, this uniformity was obtained through a further partitioning of the cell conditions $C_i$. Unfortunately, the cost of this (especially for $m$) is that the order of the associated multi-balls $\Sigma_i$ may increase, and hence we risk losing the minimality. The proof of the following lemma illustrates that this can indeed happen.

\begin{lemma}\label{lem:maximal} Let $C^\Sigma$ be a regular clustered cell of minimal order over $S$. Then for every $s\in S$, every ball $B$ which is an equivalence class of $(C,\Sigma_s)$, is maximally contained in $\Sigma_s$. 
\end{lemma}
\begin{proof}
Note that, if $C^\Sigma$ is defined by a large cell condition, then $\Sigma$ already satisfies the conclusion of this lemma. Indeed, regularity implies that all branching heights are below $\alpha(s)$, and hence the equivalence classes are always maximal balls. 

Suppose now that  $C^{\Sigma}$ is a small clustered cell of (minimal) order $k >1$, and that the maximal balls of $\Sigma_s$ contain more than one $(C,\Sigma_s)$-equivalence class. We need to show that $C^{\Sigma}$ cannot be of minimal order. We will do this by showing explicitly how to decompose $C^{\Sigma}$ as a finite union of regular clustered cells of order strictly smaller than $k$. 
\\\\
 So let $C^\Sigma$ be a clustered cell associated to a small cell condition $C$ with its leaf at height $\gamma(s)$. Because of Lemma \ref{lem:sym}, 
we know that all elements of $(\Sigma)_s$ have the same $N$-signature for all $N \in \N$, uniformly in $S$.  Hence, we can assume that there exists some $\ell \in \Z$ with $\ell < m$, such that for all $s\in S$, $\Sigma_s$ consists of (maximal) balls of the same size $\gamma(s)+\ell$  (here we also use our assumption that maximal balls contain more than one equivalence class). 

First consider the case where $\ell \leqslant 0$. Consider a maximal ball $B$ in $\Sigma_s$. We claim that $C^{B} = B$. 
Take $b \in B$. Then $B$ contains an element $c$ with $\ord(c-b) = \gamma(s)$ and $\acm(c-b) = \lambda$, and hence $B \subset C^{B}$. The other inclusion is proven similarly. Hence, both $\Sigma_s$ and $C^{\Sigma_s}$ consist of $k':= k/q_K^{m-\ell}$ maximal balls. We will rewrite both the cell condition and the set of centers, such that these $k'$ balls become the leaves of the new small cell fibers.  Write $\rho(s) =\gamma(s)+\ell$ for the size of the maximal balls in $\Sigma$. First put
\[ \widetilde{\Sigma}:= \{(s,c') \in S\times K \mid \exists (s,c) \in \Sigma : \ord(c'-c) = \rho(s)-1 \wedge c-c' \in Q_{1,1}\},\]
and let $\widetilde{C}$ be the cell condition
\[ \widetilde{C}(s,c,t):= s\in S \wedge \rho(s)-1 = \ord(t-c) \wedge t-c \in Q_{1,1}.\]
Then clearly, $\widetilde{C}^{\widetilde{\Sigma}}$ is a regular clustered cell defining the same set as $C^{\Sigma}$, yet having strictly smaller order.

For $0<\ell<m$, we can apply the inverse operation of the repartitioning of Lemma-Definition 5.1 from \cite{cham-cub-leen}, part (c). There we observed that, when increasing the value of $m$ (in the set $Q_{n,m}$), the effect was that a single equivalence class was split in smaller equivalence classes. In our case, we will replace the original condition $\ac_m(t-c) \in \lambda Q_{n,m}$ in $C$ by a condition $\ac_{\ell}(t-c) \in \lambda Q_{n,\ell}$, and call the resulting cell condition $\widehat{C}$. Then $\widehat{C}^\Sigma$ will be a clustered cell where the maximal balls of $\Sigma_s$ coincide with the $(\widehat{C}, \Sigma_s)$-equivalence classes. Since moreover, the order of $\widehat{C}^\Sigma$ is smaller than the order of  $C^\Sigma$, this completes the proof.  
\end{proof} 


\section{Equivalent questions} \label{sec:questions}

Let us first introduce the following definition. Here $\mathbb{B}$ denotes the set of balls in $K$.

\begin{definition}
Let $\Sigma$ be a multi-ball of order $k$ over $S$. 
\item We say that $\Sigma$ is \emph{$\subseteq$-maximal} if for every $s\in S$, every ball $B$ among the $k$ balls whose union is $\Sigma_s$, is maximal with respect to inclusion in $\Sigma_s$.
\item We say that a $\subseteq$-maximal multi-ball $\Sigma$ \emph{admits finite Skolem functions} if there exists a definable function $f:S\to \mathbb{B}$ such that for all $s\in S$, $f(s)$ is a maximal ball of $\Sigma_s$.
\end{definition}
When we say that a function $f:S\to \mathbb{B}$ is definable, we simply mean that its graph should correspond to a definable set $H\subseteq S\times K$, such that $H_s$ is a ball for all $s\in S$. 

\begin{lemma}\label{question} The following questions are equivalent:
\begin{enumerate}
\item \label{qq1}  Can every regular clustered cell of finite order be partitioned into finitely many regular clustered cells of order 1?
\item \label{q2} Is every regular clustered cell of minimal order of order 1? 
\item \label{q3} Does every $\subseteq$-maximal multi-ball admit finite Skolem functions?
\end{enumerate}
\end{lemma}
\begin{proof}
We first show that Questions \ref{qq1} and \ref{q2} are equivalent. By Lemma \ref{lem:decomMinimal}, if the answer to Question \ref{q2} is yes, then Question \ref{qq1} has a positive answer as well.
 Now suppose that Question \ref{qq1} has an affirmative answer and let $C^\Sigma$ be a regular clustered cell of minimal order $k>1$. By assumption, it is equal to a finite union of regular clustered cells of order 1, contradicting the minimality of the order of $C^\Sigma$. Hence Question \ref{q2} has affirmative answer too.
 
Now let us show that Questions  \ref{q2} and \ref{q3} are equivalent as well. Suppose that Question \ref{q3} has an affirmative answer. Let $C^\Sigma$ be a regular clustered cell of minimal order $k>1$. By Lemma \ref{lem:maximal}, we may assume that $\Sigma$ is $\subseteq$-maximal. Pick a finite Skolem function for $\Sigma$, say with graph $H \subseteq S \times K$. Then we have that $C^\Sigma$ is equal to the union of $C^{H}$ and $C^{\Sigma\setminus H}$. Since $C^{H}$ has order 1 and $C^{\Sigma\setminus H}$ has order $k-1$, this contradicts the minimality of $k$. 

Now suppose that every regular clustered cell of minimal order has order 1 (i.e., Question \ref{q2} has an affirmative answer) and let $\Sigma$ be a $\subseteq$-maximal multi-ball of order $k>1$ over $S$. Note that by the definition of multi-balls, these $k$ maximal balls have the same radius $\gamma(s)$ for every $s \in S$. 
Using Theorem \ref{thm:clustered2}, the set $\Sigma$ can be partitioned as a finite union of classical cells $D_i$ and regular clustered cells $C_i^{\Sigma_i}$ of minimal order, i.e.
\[\Sigma = \bigcup_{i=1}^{r_1} D_i \cup \bigcup_{i=1}^{r_2} C_i^{\Sigma_i}.\] 
By our assumptions, the cells $C_i^{\Sigma_i}$ all have order 1. Without loss of generality, we may also assume that all cells in the decomposition are over $S$. 
Now put
\begin{align*}
 \gamma_{D_i}(s)&:= \min\{\gamma \in K \mid (D_i)_s \text{ contains a ball of radius } \gamma\},\\
 \gamma_{C_i}(s)&:= \min\{\gamma \in K \mid (C_i^{\Sigma_i})_s \text{ contains a ball of radius } \gamma\}.
 \end{align*}
We will first explain why $k>1$ implies that $r_1 + r_2 >1$. Suppose that $r_1 + r_2 =1$. We will only consider the case $(r_1, r_2) = (0,1)$, but the other case is completely similar. In this case, we have that $\Sigma_s = (C_1^{\Sigma_1})_s$ for all $s \in S$. 
 However, note that $\Sigma_s$ is the union of $k >1$ disjoint maximal balls of the same size, while the fiber $(C_1^{\Sigma_1})_s$ can only contain a single ball of any given radius. This gives a contradiction, and hence it must be the case that $r_1 + r_2>1$.
 
  Moreover, since $\Sigma$ is $\subseteq$-maximal, we may assume that $\gamma(s)\leqslant \gamma_{D_i}(s)$ and $\gamma(s)\leqslant \gamma_{C_i}(s)$. Write $D_i^{\gamma_{D_i}}$, resp.\ $C_i^{\gamma_{C_i}}$ for the subset of $D_i$, resp.\ $C_i^{\Sigma_i}$ whose fibers are the (unique) maximal balls of $(D_i)_s$, resp. $(C_i^{\Sigma_i})_s$. 
  If $r_1 \neq 0$, we can define $f$ as 
\[
f(s)= B \Leftrightarrow B \text{ is a ball, maximally contained in $\Sigma_s$ and } B \cap D_1^{\gamma_{D_1}}\neq \emptyset,  
\]
otherwise we put
\[
f(s)= B \Leftrightarrow B \text{ is a ball, maximally contained in $\Sigma_s$ and } B \cap C_1^{\gamma_{C_1}}\neq \emptyset.  
\]
Then $f$ provides a finite Skolem function for $\Sigma$.   
\end{proof}

An indication that the above questions may well have a negative answer comes from Remark 4.8 of \cite{HMRC-2015}. In this remark, it is shown that there exist elementary extensions $K$ of $\QQ_p$ (for the language of rings), in which the set of balls is not rigid. In particular, the authors show that there exists an automorphism $\sigma$ of $K$ and a ball $B$ such that the orbit of $B$ under $\sigma$ has size $p$. 
This seems to imply that the answer to the questions in Lemma \ref{question} would be no, at least if there exists such a set of $p$ balls which  is also $\subseteq$-maximal, i.e., this set of $p$ balls does not cover a bigger ball of $K$.  

In future work, we hope to use this observation as a basis for constructing an example showing that the question in its third form has a negative answer. The basic idea is to try building a parametrized family of subsets having fibers that are such sets of $p$ balls.  However, actually constructing such a higher-order multi-ball (and proving that no finite skolem function exists), appears to be a rather non-trivial exercise. 

Part of the complication, especially for the one-sorted case, lies in the fact that one needs to work within a structure that does not admit definable Skolem functions. However, such structures have not been studied in much detail as yet, given that a first concrete example was only very recently constructed by the second author and K. Huu Nguyen \cite{cubi-nguyen}.

\section{Properties of clustered cells of minimal order} \label{sec:biggerorder}

In this section we will show different properties of regular clustered cells of minimal order. 
Note that the results of this sections are only of relevance in case Question \ref{q1} has a negative answer. Otherwise, all of these considerations, while still true, are essentially trivial, as we would only need to consider clustered cells of order 1.

\begin{lemma}\label{lemma:ac} Let $C^{\Sigma}$ be a regular clustered cell. Then $C^\Sigma$ can be expressed as the union of finitely many regular clustered cells $C_i^{\Sigma_i}$ of minimal order such that on each $\Sigma_i$, $\ac_1(c)$ is constant for all elements $c$ of $(\Sigma_i)_s$.
\end{lemma}
\begin{proof}
%
%
%


Let $C^{\Sigma}$ be a regular clustered cell of order $k$. We prove the result by induction on $k$. For $k=1$, the result holds automatically, since within a maximal ball $B$ of $\Sigma_s$, $\ac_1(c)$ will always be constant, since we know that $\ord(c)$ is constant for all $c \in \Sigma_s$. 
Indeed, for a maximal ball to contain elements $c, c'$ with $\ac_1(c) \neq \ac_1(c')$, would imply that this is a ball around 0. But since also $\ord(c)$ is constant, this would imply that $B = \{0\}$, contradicting the assumption that $C^{\Sigma}$ is a regular clustered cell.

For the inductive case, first partition $S$ as $S_1 \cup S_2$, where $s \in S_1$ if $\ac_1(c)$ is constant on $\Sigma_s$, and $S_2 = S \backslash S_1$. Put $\Sigma_1: = \Sigma_{|S_1}$ and $\Sigma_2: = \Sigma_{|S_2}$. For $C^{\Sigma_1}$, either this is already a clustered cell of minimal order, in which case we are done, or $C^{\Sigma_1}$ can be partitioned into regular clustered cells of strictly lower order, in which case the result follows by induction. 

Next, for $C^{\Sigma_2}$, we can do the following.
Partitioning $S$ if necessary, we may assume that the set $\ac_1((\Sigma_{2})_s) = \{b_1, \ldots, b_r\}$ is independent of $s$. Then, putting $\Sigma_{2,i}:= \{(s,c) \in \Sigma_2 \mid \ac_1(c) = b_i\}$, we obtain a partition of $C^{\Sigma_2}$ into regular clustered cells $C^{\Sigma_{2,i}}$. Note that an additional partitioning of $S$ may be required to restore the regularity. 
Since the order of each $C^{\Sigma_{2,i}}$ is strictly smaller than $k$, the result follows by induction.
%
%
\end{proof}

\begin{lemma}\label{lem:sym} Let $C^\Sigma$ be a regular clustered cell of minimal order $k$. Then there is $d \in \N$ such that for every $s \in S$, $\Sigma_s$ has exactly $d$ branching heights and when $d \geqslant 1$, there exist $k_1, \ldots, k_d \in \N$ such that all elements of each $\Sigma_s$ have $d$-signature $(k_1, \ldots, k_d)$.
\end{lemma}

\begin{proof} The existence of $d$ follows already from regularity. The statement is trivial if $k=1$, so we may assume that $k>1$, which in turn implies that $d\geqslant 1$. For  each $l \in \N$, write $(k_1(c), \ldots, k_l(c))$ for the $l$-signature of $c \in \Sigma_s$.
We will give a proof by contradiction, so assume that for at least some $s \in S$, the $d$-signature is not fixed on $\Sigma_s$.

Recall that we are assuming that $C^{\Sigma}$ is a regular clustered cell, meaning that the tree structure is uniform in $s$, and hence for every $s\in S$, $\Sigma_s$ will contain elements with at least two different signatures. We will show that in this situation, $C^{\Sigma}$ cannot be minimal, by giving an explicit decomposition into clustered cells of strictly lower order.
First, partition $\Sigma$ in sets $\Sigma_{(l_1)}$, for $l_1 \in \{1, \ldots, q_K\}$, which are defined as
\[\Sigma_{(l_1)}:= \{(s,c) \in \Sigma \mid k_1(c) = l_1\}.\]
Note that some of these sets may be empty. This induces a partition of $C^{\Sigma}$ into the union of the regular clustered cells $C^{\Sigma_{(l_1)}}$. (It should be clear that the uniformity of the tree structure is preserved. Further, since the tree of $(\Sigma_{(l_1)})_s$ is a pruning of the original tree of $\Sigma$, and no new branching heights are introduced, we still have that all branching happens below $\alpha(s)$.)

This process can now be repeated inductively. If we fix a clustered cell $C^{\Sigma_{(l_1)}}$, the 1-signature is fixed. This clustered cell can now be partitioned into cells $C^{\Sigma_{(l_1,l_2)}}$, where $C^{\Sigma_{(l_1,l_2)}}$ is defined as
\[\Sigma_{(l_1,l_2)}:= \{(s,c) \in \Sigma_{(l_1)} \mid k_2(c) = l_2\},\]
again for $l_2 \in \{1, \ldots, q_K\}$. Repeating the process until we have a partition into regular clustered cells $C^{\Sigma_{(l_1, \ldots, l_N)}}$ having strictly smaller order, yields the claim of the lemma. 
\end{proof}

For a regular clustered cell of minimal order $C^\Sigma$, the tuple $(k_1, \ldots, k_d)$, given by the previous lemma, will be called the \emph{tree type} of $C^\Sigma$. 

\begin{lemma}\label{lemma:pdiv}
Let $C^{\Sigma}$ be a regular clustered cell over $S$ of minimal order, with tree type $(k_1, \ldots, k_d)$. Then $p \mid k_d$. 
\end{lemma}
\begin{proof}
Suppose that $p \nmid k_d$. We will show that $C^{\Sigma}$ can be partitioned into regular clustered cells of order strictly lower than $k$, and hence is not minimal. By Lemma \ref{lemma:ac}, we may assume that  
%
the fibers $\Sigma_s$ consist of a finite number of disjoint balls $B_i$ of the same radius and with the same valuation and angular component (i.e. all elements have the same value under $\ac_1$). 
We will write $\gamma_d$ for the lowest branching height of the tree of $\Sigma_s$. At this height, the tree has only a single node, splitting in $k_d$ subtrees $T_1, \ldots, T_{k_d}$.  
 For each $s$, write $\mathcal{T}_s$ for the set
\[\mathcal{T}_s = \{ T_1, T_2,\ldots, T_{k_d}\}.\]
We need to show that for each $s$,  there exists a definable, non-empty strict subset $\mathcal{T}'_s$ of $\mathcal{T}_s$.
For a set $\Sigma_s \subseteq K$, we will say that points $(x_1, \ldots, x_r)$ are separated if they all belong to a different subtree of $\mathcal{T}_s$, i.e.
\[\text{Sep}_r(\Sigma_s) := \{(x_1, \ldots, x_r) \in (\Sigma_s)^r \mid  \ord(x_i - x_j) \leqslant \gamma_d \ \ \text{if}\ \ i \neq j \}. \]
For every $r \geqslant 1 $, let $G_{r,s} \subseteq K$ be the set defined by
\[G_{r,s}:=\left\{g \in  K \ \left| \ \exists (x_1, \ldots, x_r) \in \text{Sep}_r(\Sigma_s) : g = \frac1r\sum_i x_i \right\}\right.,\]
so that $G_{r,s}$ consists of all  possible averages of tuples of $r$ separated points in $\Sigma_s$. Note that $G_{r,s} = \emptyset$ if $r > |\mathcal{T}_s| = k_d$.  
\\\\
Now, we will consider the set $\widetilde{\Sigma}$, obtained by translating each $\Sigma_s$ over $G_{k_d,s}$, i.e. 
\[ \widetilde{\Sigma} := \{(s,y) \in S \times K \mid y \in \Sigma_s - G_{k_d,s}\}.\]
Write $\pi_{\gamma_d}$ for some arbitrary, fixed element of valuation $\gamma_d$. Then there exists $a\in K$, and, for each $i\in I:= \{1,\ldots,k_d\}$, an element $b_i\in \mathcal{O}_K$, such that the following holds. Every $x\in \Sigma_s$ whose $(C,\Sigma_s)$-equivalence class is contained in $T_i$, can be written as 
\[
x = a + \pi_{\gamma_d}(b_i+\pi \tilde{x}),\]
for some $\tilde{x}\in \mathcal{O}_K$. Further, there can be at most one $i_0 \in I$ for which $\ord\, b_{i_0}>0$. If $i,j \in I \backslash \{i_0\}$, then  $\ac_1(b_i)\neq\ac_1(b_j)$ when $i\neq j$. 
Using this representation, we get that elements $g\in G_{k_d,s}$ have the following form:
\[
g= \frac{1}{k_d}\sum_{i=1}^{k_d}(a+\pi_{\gamma_{d}}(b_i+\pi\tilde{x_i}))=a + \frac{\pi_{\gamma_d}}{k_d}\left(\sum_{i=1}^{k_d}b_i + \pi\sum_{i=1}^{k_d} \tilde{x_i}\right),
\]
for some $\tilde{x_i}\in \mathcal{O}_K$.
Hence, there is $b\in K$ such that every $g\in G_{k_d,s}$ is of the form \[g = a+\pi_{\gamma_d}(b+\pi\tilde{g}),\] for some $\tilde{g}\in K$. Furthermore, if $p \nmid \ k_d$, then $b$ and $\tilde{g}$ are elements of $\mathcal{O}_K$, 
which in turn implies that every $x-g\in \widetilde{\Sigma}_s$ is of the form
\[
x-g= \pi_{\gamma_d}(b_i-b +\pi(\tilde{x}-\tilde{g})),
\]
where $b_i$ is as before, depending only on the $(C, \Sigma_s)$-equivalence class of $x$. In particular, there can be at most one $j_0 \in I$ for which $\ord(b_{j_0}-b)>0$. If $i,j \in I \backslash \{j_0\}$, then  $\ac_1(b_i-b)\neq\ac_1(b_j-b)$ when $i\neq j$. 

Let $\widehat{\Sigma}$ be the set whose fibers $\widehat{\Sigma}_s$ are obtained by first restricting each $\widetilde{\Sigma}_s$ to the elements of minimal order, and subsequently choosing the elements having some fixed value for $\ac_1(\cdot)$.   
Now consider 
\[\widehat{\Sigma}^{-1} := \{(s,x) \in \Sigma \mid x \in \widehat{\Sigma}_s + G_{k_d,s}\}.\]
By construction, there will be some $i \in I$ 
 (depending on $s$) such that every $x\in (\widehat{\Sigma}^{-1})_s$ is of the form $x=a+\pi_{\gamma_d}(b_i+\pi\tilde{x})$. Hence, $(\widehat{\Sigma}^{-1})_s$ is contained in the ball $a+\pi_{\gamma_d}(b_i+\pi\mathcal{O}_K)$. If we now put $\Sigma':= \{ (s,x) \in \Sigma \mid \exists (s,x') \in \widehat{\Sigma}^{-1} : \ord(x-x') > \gamma_d\}$, then the tree of $\Sigma'$ is precisiely  $T_i$, which completes the proof.
\end{proof}

Note that, if $p\mid k_d$, the construction explained in the above proof may not yield anything useful, as it might happen that $\ord\, b<0$. We will illustrate this using an example.
 Remark however, that separating balls would be trivial in any structure that has definable Skolem functions. So, in order to provide a meaningful example, we will simulate a situation where Skolem functions do not exist. 

Take $K$ such that $k_K=\mathbb{F}_p$, let $S = K$, and assume there exists a definable set $\Sigma \subset S \times K$, such that 
\[\Sigma_s =  f(s) + \mathcal{O}_K,\]
where $f: S \to K$ is a non-definable function such that $\ord f(s) < 0$ for all $s \in S$. 
Then $\Sigma_s$ consists of balls $B_{i,s} = f(s) + i + \pi\mathcal{O}_K$, for $0 \leqslant i <p$. (In a clustered cell $C^{\Sigma}$, the sets $B_{i,s}$ could represent equivalence classes of centers.)

In this example, the subtrees $T_i$ are the balls $f(s) + i + \pi\mathcal{O}_K$, and we find that $G_{p,s} = f(s) + \mathcal{O}_K$, and hence $\widetilde{\Sigma_s} =  \mathcal{O}_K$, which is the union of the balls $j + \pi\mathcal{O}_K$, for $0 \leqslant j <  p$. 
The next step of the algorithm requires us to choose one of these balls (=subtrees) for $\widehat{\Sigma}_s$, say $\widehat{\Sigma}_s = j +\pi\mathcal{O}_K$. However, independent of the value of $j$, we always get that $\widehat{\Sigma}_s^{-1} = f(s) + \mathcal{O}_K$, which is the same as the original set $\Sigma_s$.

\

\bibliographystyle{plain}
\bibliography{../bibliography}

\end{document}